\documentclass[12pt,a4paper]{amsart}
\usepackage{amssymb}
\usepackage{amsthm}
\usepackage{amsmath}
\usepackage{amsxtra}
\usepackage{latexsym}
\usepackage{mathrsfs}
\usepackage{amscd}
\usepackage[all,cmtip]{xy}
\usepackage[all]{xy}
\usepackage{enumitem}
\usepackage{xcolor}
\usepackage{comment}

\newtheorem{thm}{Theorem}[section]
\newtheorem{lem}[thm]{Lemma}
\newtheorem{conj}[thm]{Conjecture}

\newtheorem{prop}[thm]{Proposition}
\theoremstyle{definition}

\newtheorem{rmk}[thm]{Remark}
\newtheorem{cor}[thm]{Corollary}
\numberwithin{equation}{section}
\newcommand{\var}{\overline}
\newcommand{\bb}{\mathbb}

\def\Q{{\mathbb Q}}
\def\R{{\mathbb R}}

\def\P{{\mathbb P}}
\def\var{\overline}

\DeclareMathOperator{\NS}{NS}

\DeclareMathOperator{\Pic}{Pic}

\DeclareMathOperator{\Aut}{Aut}

\DeclareMathOperator{\Exc}{Exc}
\DeclareMathOperator{\Supp}{Supp}
\DeclareMathOperator{\Sym}{Sym}
\DeclareMathOperator{\Ker}{Ker}
\allowdisplaybreaks[1]
\title[Degrees of endomorphisms on surfaces]
{Arithmetic degrees and dynamical degrees of endomorphisms on surfaces}
\author{Yohsuke Matsuzawa}
\author{Kaoru Sano}
\author{Takahiro Shibata}
\address{Graduate school of Mathematical Sciences, the University of Tokyo, Komaba, Tokyo,
153-8914, Japan}
\address{Department of Mathematics, Faculty of Science, Kyoto University, Kyoto 606-8502, Japan}
\address{Department of Mathematics, Faculty of Science, Kyoto University, Kyoto 606-8502, Japan}
\email{myohsuke@ms.u-tokyo.ac.jp}
\email{ksano@math.kyoto-u.ac.jp}
\email{tshibata@math.kyoto-u.ac.jp}
\begin{document}
\maketitle
\begin{abstract}
For a dominant rational self-map on a smooth projective variety
defined over a number field,
Kawaguchi and Silverman conjectured that
the (first) dynamical degree is equal to the arithmetic degree
at a rational point whose forward orbit is well-defined and Zariski dense.
We prove this conjecture for surjective  
endomorphisms on smooth projective surfaces.
For surjective endomorphisms on any smooth projective varieties,
we show the existence of rational points
whose arithmetic degrees are equal to the dynamical degree.
Moreover, we prove that
there exists a Zariski dense set of rational points having disjoint orbits if 
the endomorphism is an automorphism.
\end{abstract}
\tableofcontents
\section{Introduction}\label{intro}
Let $k$ be a number field, $X$ a smooth projective variety over $\overline{k}$,  
and $f\colon X\dashrightarrow X$ a dominant rational self-map on $X$ over $\overline{k}$. 
Let $I_f \subset X$ be the indeterminacy locus of $f$. 
Let $X_f (\overline{k})$ be the set of $\overline{k}$-rational points $P$ 
on $X$ such that $f^n(P) \notin I_f$ for every $n \geq 0$.
For $P\in X_f(\var{k}),$ its {\it forward $f$-orbit} is defined as 
$\mathcal{O}_f(P):=\{ f^n(P):n\geq 0\}.$

Let $H$ be an ample divisor on $X$ defined over $\overline{k}$.
The ({\it first}) {\it dynamical degree} of $f$ is defined by 
$$\delta_f :=\lim_{n\to \infty} ((f^n)^\ast H \cdot H^{\dim X -1})^{1/n}.$$
The first dynamical degree of a dominant rational self-map on a smooth complex projective variety was first defined by 
Dinh and Sibony in \cite{dinhsib, dinh}. In \cite{Tru}, Truong gave an algebraic definition of dynamical degrees.

The {\it arithmetic degree}, introduced by Silverman in \cite{Gm}, of $f$ at a $\overline{k}$-rational point $P\in X_f(\overline{k})$
is defined by
$$\alpha _f(P):=\lim_{n\to \infty} h_H^+(f^n(P))^{1/n}$$
if the limit on the right hand side exists.
Here, $h_H\colon X(\overline{k})\longrightarrow [0,\infty )$
is the (absolute logarithmic) Weil height function associated with $H$,
and we put $h_H^+:=\max\{ h_H,1\}$. 

Then we have two types of quantity concerned with the iteration of the action of $f$. 
It is natural to consider the relation between dynamical degrees and arithmetic degrees.
In this direction, Kawaguchi and Silverman formulated the following conjecture.

\begin{conj}[{The Kawaguchi--Silverman conjecture (see \cite[Conjecture 6]{rat})}] \label{KS} 
For every $\overline{k}$-rational point $P\in X_f(\overline{k})$,
the arithmetic degree $\alpha_f(P)$ exists.
Moreover, if the forward $f$-orbit $\mathcal{O} _f(P)$ is Zariski dense in $X$,
the arithmetic degree $\alpha _f(P)$ is equal to the dynamical degree $\delta_f$,
i.e., we have $$\alpha _f(P)=\delta _f.$$
\end{conj}

\begin{rmk}\label{rem_for_KS}
Let $X$ be a complex smooth projective variety with $\kappa(X)>0$,
$\Phi: X \dashrightarrow W$ the Iitaka fibration of $X$, and 
$f \colon X \dashrightarrow X$ a dominant rational self-map on $X$.
Nakayama and Zhang proved that there exists an automorphism 
$g \colon W \longrightarrow W$ of finite order such that 
$\Phi \circ f = g \circ \Phi$ (see \cite[Theorem A]{NaZh}).
This implies that any dominant rational self-map on a smooth projective 
variety of positive Kodaira dimension does not have a Zariski dense orbit.
So the latter half of Conjecture \ref{KS} is meaningful only for 
smooth projective varieties of non-positive Kodaira dimension.
However, we do not use their result in this paper.
\end{rmk}

When $f$ is a dominant {\it endomorphism} (i.e.~$f$ is defined everywhere),
the existence of the limit defining the arithmetic degree
 was proved in \cite{ab1}.
But in general, the convergence is not known.
It seems difficult at the moment to prove Conjecture \ref{KS} in full generality.

In this paper, we prove Conjecture \ref{KS} for any endomorphisms on any smooth projective surfaces:

\begin{thm}\label{Theorem:MainTheorem}
Let $k$ be a number field, $X$ a smooth projective surface over $\overline k$, 
and $f \colon X \longrightarrow X$ a surjective endomorphism on $X$.
Then Conjecture \ref{KS} holds for $f$.
\end{thm}

As  by-products of our arguments, we also obtain the following two cases 
for which Conjecture \ref{KS} holds:

\begin{thm}[Theorem {\ref{thm3.3.1}}]\label{Theorem:BirationalOnSurfaces}
Let $k$ be a number field, $X$ a smooth projective irrational surface over $\overline k$, 
and $f \colon X \dashrightarrow X$ a birational automorphism on $X$.
Then Conjecture \ref{KS} holds for $f$.
\end{thm}

\begin{thm}[Theorem {\ref{thm3.3.2}}]\label{Theorem:ToricEndo}
Let $k$ be a number field, 
$X$ a smooth projective toric variety over $\overline k$, 
and $f\colon X\longrightarrow X$ a  toric  surjective endomorphism on $X$.
Then Conjecture \ref{KS} holds for $f$.
\end{thm}

As we will see in the proof of Theorem \ref{Theorem:MainTheorem},
there does not always exist a Zariski dense orbit for a given self-map. 
For instance, a self-map cannot have a Zariski dense orbit
if it is a self-map over a variety of positive Kodaira dimension.
So it is also important 
to consider whether a self-map has a $\overline k$-rational point 
whose orbit has full arithmetic complexity, that is, 
whose arithmetic degree coincides with the dynamical degree.
We prove that such a point always exists for any surjective endomorphism on 
any smooth projective variety.

\begin{thm}\label{thm_existence}
Let $k$ be a number field, $X$ a smooth projective variety over $\overline k$,
and $f \colon X \longrightarrow X$ a surjective endomorphism on $X$.
Then there exists a $\overline k$-rational point $P \in X(\overline k)$ 
such that $\alpha_f(P)=\delta_f$.
\end{thm}

If $f$ is an automorphism, we can construct a  ``large'' collection of points
whose orbits have full arithmetic complexity. 

\begin{thm}\label{thm_large collection}
Let $k$ be a number field, $X$ a smooth projective variety over $\overline k$,
and $f \colon X \longrightarrow X$ an automorphism.
Then there exists a subset $S \subset X( \overline{k})$
which satisfies all of the following conditions.
\begin{enumerate}
\item[\rm (1)] For every $P \in S$, $ \alpha_{f}(P)=\delta_{f}$.
\item[\rm (2)] For $P, Q \in S$ with $P \neq Q$, $ \mathcal{O}_{f}(P) \cap \mathcal{O}_{f}(Q) = \emptyset$.
\item[\rm (3)] $S$ is Zariski dense in $X$.
\end{enumerate}
\end{thm}

\begin{rmk}\label{results}
Kawaguchi, Silverman, and the second author proved Conjecture \ref{KS} in the following cases
(for details, see \cite{ab1}, \cite{eg}, \cite{sano1}, \cite{Gm}, \cite{ab2}).
\begin{itemize} 
\item[(1)] (\cite[Theorem 2 (a)]{eg}) $f$ is an endomorphism and
		the N\'eron-Severi group of $X$ has rank one.
\item[(2)] (\cite[Theorem 2 (b)]{eg}) $f$ is the extension to $\mathbb{P} ^N$
		of a regular affine automorphism on $\mathbb{A} ^N$.
\item[(3)] (\cite[Theorem A]{surf}, \cite[Theorem 2 (c)]{eg}) $X$ is a smooth projective surface
		and $f$ is an automorphism on $X$.
\item[(4)] (\cite[Proposition 19]{Gm}) $f$ is the extension to $\mathbb{P}^N$
		of a monomial endomorphism on $\mathbb{G}_m^N$
		and $P\in \mathbb{G} _m^N(\overline{k})$.
\item[(5)] (\cite[Corollary 31]{ab1}, \cite[Theorem 2]{ab2}) $X$ is an abelian variety.
		Note that any rational map between abelian varieties is automatically a morphism.
\item[(6)] (\cite[Theorem 1.3]{sano1}) $f$ is an endomorphism and
		$X$ is the product $\prod_{i=1}^n X_i$ of smooth projective varieties, with the assumption that each variety $X_i$ satisfies one of the following conditions:
\begin{itemize}
\item the first Betti number of $(X_i)_{\bb{C}}$ is zero
		and the N\'eron--Severi group of $X_i$ has rank one,
\item $X_i$ is an abelian variety,
\item $X_i$ is an Enriques surface, or
\item $X_i$ is a $K3$ surface.
\end{itemize}
\item[(7)](\cite[Theorem 1.4]{sano1}) $f$ is an endomorphism and
		$X$ is the product $X_1\times X_2$ of positive dimensional varieties
		such that one of $X_1$ or $X_2$ is of general type.
		(In fact, there do not exist Zariski dense forward $f$-orbits on such $X_1\times X_2$.)
\end{itemize}
\end{rmk}

\subsection*{Notation}
\begin{itemize}

\item Throughout this paper, we fix a number field $k$.

\item A \textit{variety} always means an integral separated scheme
of finite type over $\overline k$ in this paper.

\item A {\it divisor} on a variety $X$ means a divisor on $X$ defined over $\overline{k}$.

\item An \textit{endomorphism} on a variety $X$ means a morphism from $X$ to itself defined over $\overline{k}$.
A \textit{non-trivial endomorphism} is a surjective endomorphism which is not an 
automorphism.

\item A \textit{curve} (resp.~\textit{surface}) simply means a smooth projective variety 
of dimension 1 (resp.~dimension 2) unless otherwise stated.

\item For any  curve $C$, the genus of $C$ is denoted by $g(C)$.

\item When we say that $P$ is a point of $X$ or write as $P\in X$,
it means that $P$ is a $\var{k}$-rational point of $X$.

\item The N\'eron--Severi group of a smooth projective variety $X$ is denoted by $\NS(X)$.
It is well-known that $\NS(X)$ is a finitely generated abelian group.
We put $\NS(X)_\bb{R} := \NS(X)\otimes_\bb{Z} \bb{R}.$

\item The symbols $\equiv$, $\sim$, $\sim_\Q$ and $\sim_{\R}$ mean 
algebraic equivalence, linear equivalence, $\Q$-linear equivalence, and $\R$-linear equivalence, respectively.

\item Let $X$ be a smooth projective variety and 
$f \colon X \dashrightarrow X$ a dominant rational self-map.
A point $P \in X_f(\overline k)$ is called \textit{preperiodic} 
if the forward $f$-orbit $\mathcal O_f(P)$ of $P$ is a finite set.
This is equivalent to the condition that $f^n(P)=f^m(P)$ for some $n, m \geq 0$ with 
$n \neq m$.

\item Let $f$, $g$ and $h$ be  real-valued functions on a domain $S$.
The equality $f = g + O(h)$ means that there is a positive constant $C$ such that 
$|f(x)-g(x)| \leq C |h(x)|$ for every $x \in S$.
The equality $f=g + O(1)$ means that there is a positive constant $C'$ such that 
$|f(x)-g(x)| \leq C'$ for every $x \in S$.

\end{itemize}

\subsection*{Outline of this paper}
In Section \ref{Section:Recall}, we recall the definitions and some properties
of dynamical and arithmetic degrees.
In Section \ref{Section:Reductions}, at first we recall some lemmata about
reduction for Conjecture \ref{KS}, which were proved in \cite{sano1} and \cite{ab2}.
Then, we prove the birational invariance of arithmetic degree,
and prove Theorem \ref{Theorem:BirationalOnSurfaces} 
and Theorem \ref{Theorem:ToricEndo}. 
In Section \ref{Section:EndomorphismsOnSurfaces}, we reduce  
Theorem \ref{Theorem:MainTheorem} to three cases, i.e. the case of $\P^1$-bundles, hyperelliptic surfaces, and surfaces of Kodaira dimension one.
In Section \ref{Section:RuledSurface}
we recall fundamental properties of $\P^1$-bundles over curves.
In Section \ref{Section:KSCforRuled}, Section \ref{Section:HyperEllipticSurface},
and Section \ref{Section:EllipticSurface},
we prove Theorem \ref{Theorem:MainTheorem} in each case
explained in Section \ref{Section:EndomorphismsOnSurfaces}.
Finally, in Section \ref{Section:ExistenceOfOrbits}, 
we prove Theorem \ref{thm_existence} and Theorem \ref{thm_large collection}.

\section{Dynamical degree and Arithmetic degree}\label{Section:Recall}
Let $H$ be an ample divisor on a smooth projective variety $X$.
The ({\it first}) {\it dynamical degree} of a dominant rational self-map
$f\colon X \dashrightarrow X$
is defined by
$$\delta_f :=\lim_{n\to \infty} ((f^n)^\ast H \cdot H^{\dim X-1})^{1/n}.$$
The limit defining $\delta_f$ exists,
and $\delta _f$ does not depend on the choice of $H$ (see \cite[Corollary 7]{dinh}, \cite[Proposition 1.2]{guedj}).
Note that if $f$ is an endomorphism, we have $(f^n)^{\ast}=(f^\ast )^n$
as a linear self-map on $\NS (X)$.
But if $f$ is merely a rational self-map,
then $(f^n)^{\ast}\neq (f^\ast )^n$ in general.

\begin{rmk}[{\cite[Proposition 1.2 (iii)]{dinh}, \cite[Remark 7]{rat}}] \label{n-th power of delta}
Let $\rho ((f ^n)^\ast)$ be the spectral radius of the linear self-map
$(f^n)^\ast \colon \NS (X)_\mathbb{R} \longrightarrow \NS (X)_\mathbb{R}.$
The dynamical degree $\delta_f$ is equal to the limit
$\lim_{n\to \infty} (\rho ((f ^n)^\ast ))^{1/n}.$
Thus we have $\delta_{f^n}=\delta_f^n$ for every $n\geq 1$.
\end{rmk}

Let $X_f(\var{k})$ be the set of points $P$ on $X$
such that $f$ is defined at $f^n(P)$ for every $n \geq 0.$
The {\it arithmetic degree} of $f$ at a point
$P\in X_f(\overline{k})$
is defined as follows.
Let $$h_H\colon X(\overline{k})\longrightarrow [0,\infty )$$
be the (absolute logarithmic) Weil height function associated with $H$
(see \cite[Theorem B3.2]{HS}).
We put
$$h_H^+(P):=\max\left\{ h_H(P),1\right\}.$$
We call
\begin{align*}
\overline{\alpha} _f(P)&:=\limsup_{n\to \infty} h_H^+(f^n(P))^{1/n}\text{ and}\\
\underline{\alpha} _f(P)&:=\liminf_{n\to \infty} h_H^+(f^n(P))^{1/n}\\
\end{align*}
{\it the upper arithmetic degree} and {\it the lower arithmetic degree} of $f$ at $P$, respectively.
It is known that $\overline{\alpha}_f(P)$ and $\underline{\alpha}_f(P)$
do not depend on the choice of $H$ (see \cite[Proposition 12]{rat}).
If $\var{\alpha}_f(P)=\underline{\alpha}_f(P)$, the limit
$$\alpha _f(P):=\lim_{n\to \infty} h_H^+(f^n(P))^{1/n}$$
is called {\it the arithmetic degree of} $f$ {\it at} $P$.

\begin{rmk}
Let $D$ be a divisor on $X$, $H$ an ample divisor on $X$,
and $f$ a dominant rational self-map on $X$.
Take $P \in X_f(\overline k)$.
Then we can easily check that
\begin{align*}
\overline{\alpha}_f(P)&\geq \limsup_{n\to\infty}h_D^+(f^n(P))^{1/n}, \text{ and}\\
\underline{\alpha}_f(P)&\geq \liminf_{n\to\infty}h_D^+(f^n(P))^{1/n}.
\end{align*}
So when these limits exist, we have
\begin{align*}
\alpha_f(P)&\geq \lim_{n\to\infty}h_D^+(f^n(P))^{1/n}.
\end{align*}
\end{rmk}

\begin{rmk}\label{convergence}
When $f$ is an endomorphism, the existence of the limit
defining the arithmetic degree $\alpha_f(P)$ was proved
by Kawaguchi and Silverman in \cite[Theorem 3]{ab1}.
But it is not known in general.
\end{rmk}

\begin{rmk}\label{upperineq}
The inequality $\overline{\alpha}_f(P)\leq \delta_f$
was proved by Kawaguchi and Silverman, and the third author
(see \cite[Theorem 4]{rat},\cite[Theorem 1.4]{Matsuzawa}).
Hence, in order to prove Conjecture \ref{KS}, it is enough
to prove the opposite inequality $\underline{\alpha}_f(P)\geq \delta_f$.
\end{rmk}

\section{Some reductions for Conjecture \ref{KS}}\label{Section:Reductions}

\subsection{Reductions}\label{Subsection:Reductions}
We recall some lemmata which
are useful to reduce the proof of some cases of Conjecture \ref{KS}
to easier cases.

\begin{lem}\label{Lemma:iterate}
Let $X$ be a smooth projective variety 
and $f\colon X \longrightarrow X$ a surjective endomorphism. 
Then Conjecture \ref{KS} holds for $f$
if and only if Conjecture \ref{KS} holds for $f^t$ for some $t\geq 1$.
\end{lem}
\begin{proof}
See \cite[Lemma 3.3]{sano1}.
\end{proof}

\begin{lem}[{\cite[Lemma 6]{ab2}}]\label{Lemma:ReductionByFiniteMorphisms}
Let $\psi \colon X \longrightarrow Y$ be a finite surjective morphism
between smooth projective varieties.
Let $f_X\colon X\longrightarrow X$ and $f_Y\colon Y\longrightarrow Y$ be
surjective endomorphisms on $X$ and $Y$, respectively.
Assume that $\psi\circ f_X=f_Y\circ\psi$.
Then Conjecture \ref{KS} holds for $f_X$
if and only if Conjecture \ref{KS} holds for $f_Y$.
\end{lem}

\begin{proof}
Since $\psi$ is a finite surjective morphism, we have $\dim X=\dim Y$.
For a point $P\in X(\overline k)$,
the forward $f_X$-orbit $\mathcal{O}_{f_X}(P)$ is Zariski dense in $X$
if and only if the forward $f_Y$-orbit $\mathcal{O}_{f_Y}(\psi(P))$ is Zariski dense in $Y$.
Let $H$ be an ample divisor on $Y$. Then $\psi^\ast H$ is an ample divisor on $X$.
Hence, we can calculate the dynamical degree and the arithmetic degree of $f_X$ as follows:
\begin{align*}
\delta_{f_X}
&= \lim_{n\to\infty} ((f_X^n)^\ast \psi^\ast H \cdot (\psi^\ast H)^{\dim X-1})^{1/n}\\
&= \lim_{n\to\infty} (\psi^\ast (f_Y^n)^\ast H \cdot (\psi^\ast H)^{\dim Y-1})^{1/n}\\
&= \lim_{n\to\infty} (\deg(\psi)((f_Y^n)^\ast H \cdot H^{\dim Y-1}))^{1/n}\\
&=\delta_{f_Y}.\\
\alpha_{f_X}(P)
&= \lim_{n\to\infty} h^+_{\psi^\ast H} (f_X^n(P))^{1/n}\\
&= \lim_{n\to\infty} h^+_H (\psi \circ f_X^n(P))^{1/n}\\
&= \lim_{n\to\infty} h^+_H (f_Y^n\circ \psi (P))^{1/n}\\
&= \alpha_{f_Y}(\psi(P)).
\end{align*}
Our assertion follows from these calculations.
\end{proof}

\subsection{Birational invariance of the arithmetic degree}
We show that  arithmetic degree is invariant under birational conjugacy. 
\begin{lem}\label{lem4.2.1}
Let $\mu \colon X \dashrightarrow Y$ be a birational map of smooth projective varieties. 
Take Weil height functions $h_X, h_Y$ associated with ample divisors $H_X, H_Y$ on $X, Y$, 
respectively. Then there are constants $M \in \mathbb R_{>0}$ and $M' \in \mathbb R$ 
such that 
$$h_X(P) \geq M h_Y(\mu(P))+M'$$
for any $P \in X(\overline k) \setminus I_\mu(\overline k)$. 
\end{lem}
\begin{proof}
Take a smooth projective variety $Z$ and  
a birational morphism $p \colon Z \longrightarrow X$ such that 
$p$ is isomorphic over $X \setminus I_\mu$ 
and $q=\mu \circ p \colon Z \longrightarrow Y$ is a morphism. 
Set $E=p^*p_*q^*H_Y-q^*H_Y$. 
Then $E$ is a $p$-exceptional divisor on $Z$ such that $-E$ is $p$-nef. 
By the negativity lemma (cf.~\cite[Lemma 3.39]{KoMo}), 
$E$ is an effective and $p$-exceptional divisor on $Z$. 
Take a sufficiently large integer $N$ such that $NH_X-p_*q^*H_Y$ is very ample. 
Then, for $P \in X \setminus I_\mu$, we have 
\begin{align*}
h_X(P) 
&= \frac{1}{N} (h_{NH_X-p_*q^*H_Y}(P)+h_{p_*q^*H_Y}(P))+O(1) \\
&\geq \frac{1}{N} h_{p_*q^*H_Y}(P)+O(1) \\ 
&= \frac{1}{N} h_{p^*p_*q^*H_Y}(p^{-1}(P)) +O(1) \\
&= \frac{1}{N} h_{q^*H_Y}(p^{-1}(P))+h_E(p^{-1}(P))+O(1) \\
&= \frac{1}{N} h_Y(\mu(P)) +h_E(p^{-1}(P))+O(1). 
\end{align*}
We know that $h_E \geq O(1)$ on $Z \setminus \Supp E$ 
(cf.~\cite[Theorem B.3.2(e)]{HS}). 
Since $\Supp E \subset p^{-1}(I_\mu)$, 
$h_E(p^{-1}(P)) \geq O(1)$ for $P \in X \setminus I_\mu$. 
Eventually, we obtain that $h_X(P) \geq (1/N)h_Y(\mu(P))+O(1)$ for 
$P \in X \setminus I_\mu$. 
\end{proof}

\begin{thm}\label{Theorem:BirationalInvariance}
Let $f \colon X \dashrightarrow X$ and 
$g \colon Y \dashrightarrow Y$ be dominant rational self-maps 
on smooth projective varieties and $\mu \colon X \dashrightarrow Y$ a birational map 
such that $g \circ \mu = \mu \circ f$. 
\begin{itemize}
\item[{\rm (i)}]
Let $U \subset X$ be a Zariski open subset such that 
$\mu|_U \colon U \longrightarrow \mu(U)$ is an isomorphism. 
Then $\overline \alpha_f(P)=\overline \alpha_g(\mu(P))$ and 
$\underline \alpha_f(P)=\underline \alpha_g(\mu(P))$ 
for $P \in X_f(\overline k) \cap \mu^{-1}(Y_g(\overline k))$ 
such that $\mathcal O_f(P) \subset U(\overline k)$. 
\item[{\rm (ii)}]
Take $P \in X_f(\overline k) \cap \mu^{-1}(Y_g(\overline k))$. 
Assume that 
$\mathcal O_f(P)$ is Zariski dense in $X$ and both $\alpha_f(P)$ and $\alpha_g(\mu(P))$ exist. 
Then $\alpha_f(P)=\alpha_g(\mu(P))$. 
\end{itemize}
\end{thm}

\begin{proof}
(i) 
Using Lemma \ref{lem4.2.1} for both $\mu$ and $\mu^{-1}$, there are constants 
$M_1, L_1 \in \mathbb R_{>0}$ and $M_2, L_2 \in \mathbb R$ such that 
\begin{equation}
M_1h_Y(\mu(P))+M_2 \leq h_X(P) \leq L_1h_Y(\mu(P))+L_2 \tag{$*$}
\end{equation}
for $P \in U(\overline k)$.  The claimed equalities follow from ($*$). 

(ii) 
Since $\mathcal O_f(P)$ is Zariski dense in $X$, we can take a subsequence $\{f^{n_k}(P)\}_k$ of 
$\{f^n(P)\}_n$ contained in $U$. 
Using ($*$) again, it follows that 
$$\alpha_f(P)=\lim_{k \to \infty}h_X^+(f^{n_k}(P))^{1/n_k}
=\lim_{k \to \infty}h_Y^+(g^{n_k}(\mu(P)))^{1/n_k}=\alpha_g(\mu(P)).$$
\end{proof}

\begin{rmk}
In \cite{Gm}, Silverman dealt with a height function on $\bb{G}_m^n$
induced by an open immersion $\bb{G}_m^n\hookrightarrow \P^n$
and proved Conjecture \ref{KS} for monomial maps on $\bb{G}_m^n$.
It seems that it had not be checked in the literature that
the arithmetic degrees of endomorphisms on 
quasi-projective varieties do not depend on
the choice of open immersions to projective varieties.
Now by Theorem \ref{Theorem:BirationalInvariance},
the arithmetic degree of a rational self-map on a quasi-projective variety at a point
does not depend on the choice of an open immersion
of the quasi-projective variety to a projective variety.
Furthermore, by the birational invariance of dynamical degrees,
we can state Conjecture \ref{KS} for rational self-maps on quasi-projective varieties,
such as semi-abelian varieties.
\end{rmk}

\subsection{Applications of the birational invariance}
In this subsection, we prove Theorem \ref{Theorem:BirationalOnSurfaces} and
Theorem \ref{Theorem:ToricEndo} as applications of Theorem 
\ref{Theorem:BirationalInvariance}.

\begin{thm}[Theorem \ref{Theorem:BirationalOnSurfaces}]\label{thm3.3.1}
Let $X$ be an irrational surface 
and $f \colon X \dashrightarrow X$ a birational automorphism on $X$.
Then Conjecture \ref{KS} holds for $f$.
\end{thm}

\begin{proof}
Take a point $P \in X_f(\overline k)$.
If $\mathcal O_f(P)$ is finite, the limit $\alpha_f(P)$ exists and is equal to 1.
Next, assume that the closure $\overline{\mathcal O_f(P)}$ of  $\mathcal O_f(P)$ 
has dimension 1.
Let $Z$ be the normalization of $\overline{\mathcal O_f(P)}$ and 
$\nu \colon Z \longrightarrow X$ the induced morphism.
Then an endomorphism $g \colon Z \longrightarrow Z$ satisfying $\nu \circ g=f \circ \nu$ 
is induced. 
Take a point $P' \in Z$ such that $\nu(P')=P$.
Then $\alpha_g(P')=\alpha_f(P)$ since $\nu$ is finite.
It follows from \cite[Theorem 2]{ab1} that $\alpha_g(P')$ exists 
(note that \cite[Theorem 2]{ab1} holds for possibly non-surjective endomorphisms on 
possibly reducible normal varieties).
Therefore $\alpha_f(P)$ exists.

Assume that $\mathcal O_f(P)$ is Zariski dense.
If $\delta_f=1$, then $1 \leq \underline \alpha_f(P) \leq \overline \alpha_f(P) \leq \delta_f 
=1$ by Remark \ref{upperineq}, so $\alpha_f(P)$ exists and $\alpha_f(P)=\delta_f=1$.
So we may assume that $\delta_f >1$.
Since $X$ is irrational and $\delta_f >1$, $\kappa(X)$ must be non-negative 
(cf.~\cite[Theorem 0.4, Proposition 7.1 and Theorem 7.2]{defa}).
Take a birational morphism $\mu \colon X \longrightarrow Y$ 
to the minimal model $Y$ of $X$ and 
let $g \colon Y \dashrightarrow Y$ be the birational automorphism on $Y$ defined as 
$g=\mu \circ f \circ \mu^{-1}$. 
Then $g$ is in fact an automorphism since, if $g$ has indeterminacy, 
$Y$ must have a $K_Y$-negative curve. 
It is obvious that $\mathcal O_g(\mu(P))$ is also Zariski dense in $Y$. 
Since $\mu(\Exc(\mu))$ is a finite set, there is a positive integer $n_0$ such that 
$\mu(f^n(P))= g^n(\mu(P)) \not\in \mu(\Exc(\mu))$ for $n \geq n_0$. 
So we have  $f^n(P) \not\in \Exc(\mu)$ for $n \geq n_0$.
Replacing $P$ by $f^{n_0}(P)$, we may assume that 
$\mathcal O_f(P) \subset X \setminus \Exc(\mu)$. 
Applying Theorem \ref{Theorem:BirationalInvariance} (i) to $P$, 
it follows that $\alpha_f(P)=\alpha_g(\mu(P))$. 
We know that $\alpha_g(\mu(P))$ exists since $g$ is a morphism.
So $\alpha_f(P)$ also exists. 
The equality $\alpha_g(\mu(P))=\delta_g$ holds as a consequence of 
Conjecture \ref{KS} for automorphisms on surfaces 
(cf.~Remark \ref{results} (3)). 
Since dynamical degree is invariant under birational conjugacy, 
it follows that $\delta_g=\delta_f$. 
So we obtain the equality $\alpha_f(P)=\delta_f$. 
\end{proof}

\begin{thm}[Theorem \ref{Theorem:ToricEndo}]\label{thm3.3.2}
Let $X$ be a smooth projective toric variety  
and $f\colon X\longrightarrow X$ a  toric  surjective endomorphism on $X$.
Then Conjecture \ref{KS} holds for $f$.
\end{thm}

\begin{proof}
Let $\mathbb G_m^d \subset X$ be the torus 
embedded as an open dense subset in $X$. 
Then $f|_{\mathbb G_m^d} \colon \mathbb G_m^d \longrightarrow \mathbb G_m^d$ is a 
homomorphism of algebraic groups by assumtion. 
Let $\mathbb G_m^d \subset \mathbb P^d$ 
be the natural embedding of $\mathbb G_m^d$ to the projective space $\mathbb P^d$ 
and $g \colon \mathbb P^d \dashrightarrow \mathbb P^d$ be the induced rational self-map. 
Then $g$ is a monomial map. 

Take $P \in X(\overline k)$ such that $\mathcal O_f(P)$ is Zariski dense. 
Note that $\alpha_f(P)$ exists since $f$ is a morphism. 
Since $\mathcal O_f(P)$ is Zariski dense and $f(\mathbb G_m^d) \subset \mathbb G_m^d$, 
there is a positive integer $n_0$ such that 
$f^n(P) \in \mathbb G_m^d$ for $n \geq n_0$. 
By replacing $P$ by $f^{n_0}(P)$, we may assume that 
$\mathcal O_f(P) \subset \mathbb G_m^d$. 
Applying Theorem \ref{Theorem:BirationalInvariance} (i) to $P$, 
it follows that $\alpha_f(P)=\alpha_g(P)$. 

The equality $\alpha_g(P)=\delta_g$ holds as a consequence of 
Conjecture \ref{KS} for monomial maps 
(cf.~Remark \ref{results} (4)).   
Since dynamical degree is invariant under birational conjugacy, 
it follows that $\delta_g=\delta_f$. 
So we obtain the equality $\alpha_f(P)=\delta_f$. 
\end{proof}

\section{Endomorphisms on surfaces}\label{Section:EndomorphismsOnSurfaces}

We start to prove Theorem \ref{Theorem:MainTheorem}.
Since Conjecture \ref{KS} for automorphisms on surfaces is already proved 
by Kawaguchi (see Remark \ref{results} (3)),
it is sufficient to prove Theorem \ref{Theorem:MainTheorem} for \textit{non-trivial} 
endomorphisms, that is, surjective endomorphisms which are not automorphisms. 

Let $f \colon X \longrightarrow X$ be a non-trivial endomorphism on a surface. 
First we divide the proof of Theorem \ref{Theorem:MainTheorem} according to 
the Kodaira dimension of $X$.

(I) $\kappa(X)=-\infty$;  
we need the following result due to Nakayama.

\begin{lem}[cf.~{\cite[Proposition 10]{Nakayama}}]\label{lem_inv}
Let $f \colon X \longrightarrow X$ be a non-trivial endomorphism 
on a surface $X$ with $\kappa(X)=-\infty$. 
Then there is a positive integer $m$ such that $f^m(E)=E$ for any irreducible curve 
$E$ on $X$ with negative self-intersection. 
\end{lem}
\begin{proof}
See {\cite[Proposition 10]{Nakayama}}.
\end{proof}

Let $\mu \colon X \longrightarrow X'$ be the contraction of a $(-1)$-curve $E$ on $X$.
By Lemma \ref{lem_inv}, there is a positive integer $m$ such that $f^m(E)=E$. 
Then $f^m$ induces an endomorphism $f' \colon X' \longrightarrow X'$ 
such that $\mu \circ f^m= f' \circ \mu$. 
Using Lemma \ref{Lemma:iterate} and Theorem \ref{Theorem:BirationalInvariance}, 
the assertion of Theorem \ref{Theorem:MainTheorem} 
for $f$ follows from that for $f'$. 
Continuing this process, we may assume that $X$ is relatively minimal. 

When $X$ is irrational and relatively minimal, 
$X$ is a $ {\mathbb{P}}^{1}$-bundle over a curve $C$ with $g(C) \geq 1$. 

When $X$ is rational and relatively minimal,
$X$ is isomorphic to $\mathbb P^2$ or
the Hirzebruch surface $\mathbb F_n= \mathbb P(\mathcal O_{\mathbb P^1} \oplus \mathcal O_{\mathbb P^1}(-n))$ for some $n \geq 0$ with $n \neq 1$. 
Note that Conjecture \ref{KS} holds 
for surjective endomorphisms on projective spaces (see Remark \ref{results} (1)). 

(II) $\kappa(X)=0$;   
for surfaces with non-negative Kodaira dimension, we use the following result due to Fujimoto.

\begin{lem}[cf.~{\cite[Lemma 2.3 and Proposition 3.1]{Fujim}}]\label{lem_min}
Let $f \colon X \longrightarrow X$ be a non-trivial endomorphism on a surface $X$ with 
$\kappa(X) \geq 0$. 
Then $X$ is minimal and $f$ is \'etale. 
\end{lem}
\begin{proof}
See \cite[Lemma 2.3 and Proposition 3.1]{Fujim}
\end{proof}

So $X$ is either an abelian surface, a hyperelliptic surface, a K3 surface, or an Enriques surface. 
Since $f$ is \'etale, we have $\chi(X, \mathcal O_X)= \deg(f) \chi(X, \mathcal O_X)$. 
Now $\deg(f) \geq 2$ by assumption, so $\chi(X, \mathcal O_X)=0$ (cf.~\cite[Corollary 2.4]{Fujim}). 
Hence $X$ must be either an abelian surface or a hyperelliptic surface because 
K3 surfaces and Enriques surfaces have non-zero Euler characteristics. 
Note that Conjecture \ref{KS} is valid for endomorphisms on abelian varieties 
(see Remark \ref{results} (5)).

(III) $\kappa(X)=1$; 
this case will be treated in Section \ref{Section:EllipticSurface}. 

(IV) $\kappa(X)=2$; 
the following fact is well-known. 

\begin{lem}\label{Lemma:EndomorphismsOnCurvesOfGeneralType}
Let $X$ be a smooth projective variety of general type.
Then any surjective endomorphisms on $X$ are automorphisms.
Furthermore, the group of automorphisms $\Aut(X)$ on $X$ 
has finite order.
\end{lem}

\begin{proof}
See {\cite[Proposition 2.6]{Fujim}}, {\cite[Theorem 11.12]{Iitaka}}, or {\cite[Corollary 2]{Matsumura}}.
\end{proof}

So there is no non-trivial endomorphism on $X$. 
As a summary, the remaining cases for the proof of 
Theorem \ref{Theorem:MainTheorem} are the following: 
\begin{itemize}
\item
Non-trivial endomorphisms on $ {\mathbb{P}^{1}}$-bundles over a curve. 
\item
Non-trivial endomorphisms on hyperelliptic surfaces. 
\item
Non-trivial endomorphisms on surfaces of Kodaira dimension 1. 
\end{itemize}

\begin{rmk}\label{rmk_classification}
Fujimoto and Nakayama gave a complete classification of surfaces which admit
non-trivial endomorphisms 
(cf.~\cite[Theorem 1.1]{FN2}, \cite[Proposition 3.3]{Fujim}, 
\cite[Theorem 3]{Nakayama}, and \cite[Appendix to Section 4]{FN1}). 
\end{rmk}

\section{Some properties of $\P^1$-bundles over curves}\label{Section:RuledSurface}
In this section, we recall and prove some properties of $\P^1$-bundles (see \cite[Chapter V.2]{AG}, \cite{Homma1}, \cite{Homma2} for detail).
In this section, let $X$ be a $\P^1$-bundle over a curve $C$.
Let $\pi\colon X\longrightarrow C$ be the projection.

\begin{prop}\label{Proposition:StructureOfRuledSurfaces}
We can represent $X$ as $X\cong \P(\mathcal{E})$,
where $\mathcal{E}$ is a locally free sheaf of rank 2 on $C$
such that $H^0(\mathcal{E})\neq 0$
but $H^0(\mathcal{E}\otimes \mathcal{L})=0$ for all invertible sheaves 
$\mathcal{L}$ on $C$ with $\deg \mathcal{L}<0$.
The integer $e:=-\deg \mathcal{E}$ does not depend on the choice of such $\mathcal E$.
Furthermore, there is a section $\sigma\colon C\longrightarrow X$
with image $C_0$ such that $\mathcal{O}_X(C_0)\cong \mathcal{O}_X(1).$
\end{prop}
\begin{proof}
See \cite[Proposition 2.8]{AG}.
\end{proof}

\begin{lem}
The Picard group and the N\'eron--Severi group of $X$ have the structure as follows.
\begin{align*}
\Pic(X)&\cong \bb{Z} \oplus \pi^\ast \Pic(C),\\
\NS(X)&\cong \bb{Z} \oplus \pi^\ast \NS(C) \cong \bb{Z} \oplus \bb{Z}.
\end{align*}
Furthermore, the image $C_0$ of the section $\sigma\colon C\longrightarrow X$
in Proposition \ref{Proposition:StructureOfRuledSurfaces}
generates the first direct factor of $\Pic(X)$ and $\NS(X)$.
\end{lem}
\begin{proof}
See \cite[V, Proposition 2.3]{AG}.
\end{proof}

\begin{lem}\label{Lemma:FiberPreservingConditions}
Let $F\in \NS(X)$ be a fiber $\pi^{-1}(p)=\pi^\ast p$ over a point $p\in C(\var{k})$, and
$e$ the integer defined in Proposition \ref{Proposition:StructureOfRuledSurfaces}.
Then the intersection numbers of generators of $\NS(X)$ are the following.
\begin{align*}
F \cdot F&= 0,\\
F \cdot C_0&= 1,\\
C_0 \cdot C_0&= -e.
\end{align*}
\end{lem}
\begin{proof}
It is easy to see that the equalities $F \cdot F=0$ and $F \cdot C_0=1$ hold.
For the last equality, see \cite[V, Proposition 2.9]{AG}.
\end{proof}
We say that {\it $f$ preserves fibers} 
if there is an endomorphism $f_C$ on $C$ such that $\pi\circ f=f_C\circ \pi$.
In our situation, since there is a section $\sigma \colon C\longrightarrow X$,
$f$ preserves fibers if and only if,
for any point $p\in C$, there is a point $q\in C$
such that $f(\pi^{-1}(p)) \subset \pi^{-1}(q)$.

The following lemma appears in \cite[p.~18]{Amerik} in more general form.
But we need it only in the case of $\P^1$-bundles on a curve,
and the proof in general case is similar to our case.
So we deal only with the case of $\P^1$-bundle on a curve.
\begin{lem}\label{Lemma:FiberPreserving}
For any surjective endomorphism $f$ on $X$, the iterate $f^2$ preserves fibers.
\end{lem}

\begin{proof}
By the projection formula,
the fibers of $\pi\colon X \longrightarrow C$ can be characterized as connected curves
having intersection number zero with any fibers
$F_p=\pi^\ast p$, $p\in C$.
Hence, to check that the iterate $f^2$
sends fibers to fibers, it suffices to show that
$(f^2)^\ast (\pi^\ast \NS(C)_\R) = \pi^\ast\NS(C)_\R$.
Since $\pi^\ast\NS(C)_\R$ is a hyperplane in $\NS(X)_\R$
such that any divisor class $D$
from this hyperplane satisfies $D \cdot D = 0$,
its pullback $f^\ast\pi^\ast \NS(C)_\R$ is a hyperplane with the same property.
There are at most two such hyperplanes,
because the form of self-intersection $\NS(X)_\R \longrightarrow\R$ is
a quadratic form associated to the coefficients of $C_0$ and $F$.
Hence, $f^*$ fixes or interchanges them and so $(f^2)^*$ fixes them.
\end{proof}

\begin{lem}\label{Lemma:EquivalenceOfFiberPreserving}
A surjective endomorphism $f$ preserves fibers 
if and only if there exists a non-zero integer $a$ such that $f^\ast F \equiv aF$. 
Here, $F$ is the numerical class of a fiber.
\end{lem}

\begin{proof}
Assume $f^{*}F\equiv aF$.
For any point $p\in C$, we set $F_p:=\pi^{-1}(p)=\pi^\ast p$.  
If $f$ does not preserve fibers, there is a point $p \in C$
such that $f(F_{p}) \cdot F>0$.
Now we can calculate the intersection number as follows:
\begin{align*}
0&= F \cdot aF = F \cdot (f^\ast F) = F_{p} \cdot (f^\ast F)\\
&= (f_\ast F_{p}) \cdot F =\deg(f |_{F_{p}})\cdot (f(F_{p}) \cdot F) > 0.
\end{align*}
This is a contradiction. Hence $f$ preserves fibers.

Next, assume that $f$ preserves fibers.
Write $f^\ast F =aF+bC_0$.
Then we can also calculate the intersection number as follows:
\begin{align*}
b&= F \cdot (aF+bC_0) = F \cdot f^\ast F = (f_\ast F) \cdot F\\
&= \deg(f |_F)\cdot(F \cdot F) = 0.
\end{align*}
Further, by the injectivity of $f^\ast$, we have $a\neq 0$.
The proof is complete.
\end{proof}

\begin{lem}\label{Lemma:PositivityOfInvariant}
If $\mathcal{E}$ splits, i.e., if there is an invertible sheaf $\mathcal{L}$ on $C$
such that $\mathcal{E}\cong \mathcal{O}_C \oplus \mathcal{L}$,
the invariant $e$ of $X= \P (\mathcal{E})$ is non-negative.
\end{lem}
\begin{proof}
See \cite[V, Example 2.11.3]{AG}.
\end{proof}

\begin{lem}\label{Lemma:Ampleness}
Assume that $e\geq 0$.
Then for a divisor $D=aF+bC_0\in \NS(X)$, the following properties are equivalent.
\begin{itemize}
\item $D$ is ample.
\item $a>be$ and $b>0.$
\end{itemize}
\end{lem}
\begin{proof}
See \cite[V, Proposition 2.20]{AG}.
\end{proof}

We can prove a result stronger than Lemma \ref{Lemma:FiberPreserving} as follows.
\begin{lem}
Assume that $e>0$.
Then any surjective endomorphism $f\colon X\longrightarrow X$ preserves fibers.
\end{lem}
\begin{proof}
By Lemma \ref{Lemma:EquivalenceOfFiberPreserving},
it is enough to prove $f^\ast F \equiv aF$ for some integer $a>0$.
We can write $f^\ast F \equiv aF+bC_0$ for some integers $a,b\geq 0$.

Since we have
$$
a F + b C_0 = ( a - be ) F + b ( e F + C_0 )
$$
and $f$ preserves the nef cone and the ample cone,
either of the equalities $a-be=0$ or $b = 0$ holds.

We have
\begin{align*}
0&= \deg(f)(F \cdot F) = (f_\ast f^\ast F) \cdot F\\
&= (f^\ast F) \cdot (f^\ast F) = (aF+bC_0) \cdot (aF+bC_0)\\
&= 2ab-b^2e = b(2a-be).
\end{align*}
So either of the equalities $b=0$ or $2a-be=0$ holds.

If we have $b\neq 0$, we have $a-be=0$ and $2a-be=0$.
So we get $a=0$. But since $e\neq 0$, we obtain $b=0$.
This is a contradiction.
Consequently, we get $b=0$ and $f^*F \equiv aF$.
\end{proof}

\begin{lem}\label{degrees}
Fix a fiber $F=F_p$ for a point $p\in C(\var{k})$.
Let $f$ be a surjective endomorphism on $X$ preserving fibers,
$f_{C}$ the endomorphism on $C$ satisfying $\pi \circ f=f_C \circ \pi$,
$f_F:=f|_F\colon F\longrightarrow f(F)$ the restriction of $f$ to the fiber $F$.
Set $f ^\ast F\equiv aF$ and $f^\ast C_0\equiv cF+dC_0$.
Then we have $a=\deg(f_C)$, $d=\deg(f_F)$, $\deg(f)=ad$, and $\delta_f=\max\{ a,d\}$.
\end{lem}
\begin{proof}
Our assertions follow from the following equalities
of divisor classes in $\NS(X)$ and of intersection numbers:
\begin{align*}
aF &= f^\ast F = f^\ast \pi^\ast p\\
	&= \pi^\ast f_C^\ast p = \pi^\ast (\deg(f_C) p)\\
	&= \deg(f_C)\pi^\ast p = \deg(f_C)F,\\
\deg(f)F &= f_\ast f^\ast F = f_\ast f^\ast \pi^\ast p\\
					&= f_\ast \pi^\ast f_C^\ast p = f_\ast \pi^\ast (\deg(f_C)p)\\
					&= \deg(f_C) f_\ast F = \deg(f_C) \deg(f_F) f(F)\\
					&= \deg(f_C) \deg(f_F) F\\
\deg(f) &= \deg(f)C_0 \cdot F = (f_\ast f^\ast C_0) \cdot F\\
					&= (f^\ast C_0) \cdot (f^\ast F)= (cF+dC_0) \cdot aF = ad.
\end{align*}
The last assertion $\delta_f=\max\{a,d \}$ follows from the functoriality of $f^\ast$
and the equality $\delta_f=\lim_{n\to\infty}\rho((f^n)^\ast)^{1/n}$ 
(cf.~Remark \ref{n-th power of delta}).
\end{proof}

\begin{lem}\label{Lemma:CulculationOfDynamicalDegree}
Let Notation be as in Lemma \ref{degrees}. 
Assume that $e\geq 0$. 
Then both $F$ and $C_0$ are eigenvectors of 
$f^* \colon \NS(X)_\mathbb R \longrightarrow \NS(X)_\mathbb R$.
Further, if $e$ is positive, then we have $\deg(f_C)= \deg (f_F)$.
\end{lem}
\begin{proof}
Set
$f^\ast F = aF$ and $f^\ast C_0= cF+dC_0$ in $\NS(X)$.
Then we have
\begin{align*}
-ead &= -e\deg f = (f_\ast f^\ast C_0) \cdot C_0\\
		 &= (f^\ast C_0)^2 = (cF+dC_0)^2 = 2cd-ed^2.
\end{align*}
Hence, we get $c=e(d-a)/2$.
We have the following equalities in $\NS(X)$:
\begin{align*}
f^\ast(eF + C_0) = aeF + (cF+dC_0) = (ae+c)F+dC_0.
\end{align*}
By the fact that $f^\ast D$ is ample if and only if $D$ is ample,
it follows  that $eF+C_0$ is an eigenvector of $f^\ast$.
Thus, we have
\begin{align*}
de = ae+c = ae+e(d-a)/2 = e(d+a)/2.
\end{align*}
Therefore, the equality $e(d-a)=0$ holds.
So $c=e(d-a)/2=0$ holds.

Further, we assume that $e>0$.
Then it follows that $d-a=0$.
So we have $\deg(f_C)=a=d=\deg(f_F)$.
\end{proof}

The following lemma is used in Subsection \ref{Subsection:Elliptic}.

\begin{lem}\label{lem_three_sections}
Let $\mathcal{L}$ be a non-trivial invertible sheaf of degree $0$
on a curve $C$ with $g(C)\geq 1$,
$\mathcal{E}=\mathcal{O}_C\oplus \mathcal{L}$, and $X=\P(\mathcal{E})$.
Let $C_0, C_1$ be sections corresponding
to the projections
$\mathcal{E}\longrightarrow \mathcal{L}$ and $\mathcal{E}\longrightarrow \mathcal{O}_C$. 
If $\sigma\colon C \longrightarrow X$ is a section such that $(\sigma(C))^2=0$,
then $\sigma(C)$ is equal to $C_0$ or $C_1$.
\end{lem}

\begin{proof}
Note that $e=0$ in this case and thus $(C_{0}^{2})=0$.
Moreover, $ \mathcal{O}_{X}(C_{0}) \cong \mathcal{O}_{X}(1)$ and 
$ \mathcal{O}_{X}(C_{1}) \cong \mathcal{O}_{X}(1) {\otimes} \pi^{*} \mathcal{L}^{-1}$.
Let $F$ be the numerical class of a fiber.
Set $\sigma(C) \equiv aC_{0}+bF$.
Then $a=(\sigma(C)\cdot F)=1$ and $2ab=(\sigma(C)^{2})=0$.
Thus $\sigma(C)\equiv C_{0}$.
Therefore,
$ \mathcal{O}_{X}(\sigma(C)) \cong \mathcal{O}_{X}(C_{0}) {\otimes} \pi^{*} \mathcal{N}$
for some invertible sheaf $ \mathcal{N}$ of degree $0$ on $C$.
Then
\begin{align*}
0&\neq H^{0}( X,\mathcal{O}_{X}( \sigma(C)))
= H^{0}(C, \pi_{*} \mathcal{O}_{X}(C_{0}) {\otimes} \mathcal{N})\\
&= H^{0}(C,  (\mathcal{L} {\oplus} \mathcal O_C)\otimes \mathcal{N})
\end{align*}
and this implies $ \mathcal{N}\cong \mathcal{O}_{C}$ or
$ \mathcal{N} \cong \mathcal{L}^{-1}$.
Hence $ \mathcal{O}_{X}(\sigma(C))$ is isomorphic to $\mathcal{O}_{X}(C_{0})$
or $\mathcal{O}_{X}(C_{0}) {\otimes} \pi^{*} \mathcal{L}^{-1}= \mathcal{O}_{X}(C_{1})$.
Since $ \mathcal{L}$ is non-trivial,
we have $H^{0}( \mathcal{O}_{X}(C_{0}))=H^{0}( \mathcal{O}_{X}(C_{1}))=\var{k}$
and we get $\sigma(C)=C_{0}$ or $C_{1}$.
\end{proof}

\section{$\P^1$-bundles over curves}\label{Section:KSCforRuled}
In this section, we prove Conjecture \ref{KS}
for non-trivial endomorphisms on $\P^1$-bundles over curves.
We divide the proof according to the genus of the base curve.

\subsection{$\mathbb P^1$-bundles over $\mathbb P^1$}\label{Subsection:Hirzebruch}

\begin{thm}\label{prop:Hirz}
Let $\pi\colon X\longrightarrow \mathbb P^1$ 
be a $\P^1$-bundle over  $\mathbb P^1$ and 
$f \colon X \to X$ be a non-trivial endomorphism.
Then Conjecture \ref{KS} holds for $f$.
\end{thm}
\begin{proof}%[Proof of Theorem \ref{Theorem:MainTheorem} in the case of $\P^1$-bundle with g(C)=0]
Take a locally free sheaf $\mathcal E$ of rank $2$ on $\mathbb P^1$ such that
$X\cong \P(\mathcal{E})$ and $\deg\mathcal{E} =-e$ 
(cf.~Proposition \ref{Proposition:StructureOfRuledSurfaces}).
Then $\mathcal{E}$ splits (see \cite[V. Corollary 2.14]{AG}).  
When $X$ is isomorphic to $\P^1\times \P^1$, i.e.~the case of $e=0$,
the assertion holds by \cite[Theorem 1.3]{sano1}.
When $X$ is not isomorphic to $\P^1 \times \P^1$, i.e.~the case of $e>0$,
the endomorphism $f$ preserves fibers and induces an endomorphism $f_{ {\mathbb{P}}^{1}}$ on the base curve $ {\mathbb{P}}^{1}$.
By Lemma \ref{Lemma:CulculationOfDynamicalDegree},
we have $\delta_f=\delta_{f_{\P^1}}$.
Fix a point $p \in \P^1$ and set $F=\pi^\ast p$.
Let $P\in X(\var{k})$ be a point whose forward $f$-orbit is Zariski dense in $X$.
Then the forward $f_{\P^1}$-orbit of $\pi (P)$ is also Zariski dense in $\P^1$.
Now the assertion follows from the following computation.
\begin{align*}
	\alpha_{f}(P) &\geq \lim_{n\to\infty} h_F(f^n(P))^{1/n}= \lim_{n\to \infty} h_{\pi^\ast p}(f^n(P))^{1/n}\\
	&= \lim_{n\to \infty} h_{p}(\pi\circ f^n(P))^{1/n}
	= \lim_{n\to \infty} h_{p}(f_{\P^1}^n \circ \pi(P))^{1/n}
	=\delta_{f_{\P^1}}=\delta_{f}.
\end{align*}
\end{proof}

\subsection{$\mathbb P^1$-bundles over genus one curves}\label{Subsection:Elliptic}

In this subsection, we prove Conjecture \ref{KS} for any endomorphisms on a $\P^1$-bundle on a curve $C$ of genus one.

The following result is due to Amerik.
Note that Amerik in fact proved it for $\mathbb P^1$-bundles over varieties of 
arbitrary dimension (cf.~\cite{Amerik}). 

\begin{lem}[Amerik]\label{Lemma:FiniteBaseChange}
Let $X=\mathbb P(\mathcal E)$ be a $\mathbb P^1$-bundle over a curve $C$.
If $X$ has a fiber-preserving surjective endomorphism 
whose restriction to a general fiber has degree greater than 1,
then $\mathcal{E}$ splits into a direct sum of two line bundles after a finite base change.
Furthermore, if $\mathcal E$ is semistable, 
then $\mathcal{E}$ splits into a direct sum of two line bundles after an \'etale base change.
\end{lem}

\begin{proof}
See \cite[Theorem 2 and Proposition 2.4]{Amerik}.
\end{proof}

The following lemma is used when we take the base change by an \'etale cover of 
genus one curve.

\begin{lem}\label{Lemma:EndomorphismOnTrivializedFibration}
Let $E$ be a curve of genus one
with an endomorphism $f\colon E\longrightarrow E$.
If $g\colon E' \longrightarrow E$ is a finite \'etale covering of $E$,
there exists a finite \'etale covering $h\colon E''\longrightarrow E'$
and an endomorphism $f '\colon E'' \longrightarrow E''$ such that
$f \circ g\circ h=g\circ h \circ f'$. 
Furthermore, we can take $h$ as satisfying $E''=E$.
\end{lem}
\begin{proof}
At first, since $E'$ is an \'etale covering of genus one curve $E$,
$E'$ is also a genus one curve.
By fixing a rational point $p\in E'(\var{k})$ and $g(p)\in E(\var{k})$,
these curves $E$ and $E'$ are regarded as elliptic curves, and
$g$ can be regarded as an isogeny between elliptic curves.
Let $h:=\hat{g}\colon E\longrightarrow E'$ be the dual isogeny of $g$.
The morphism $f$ is decomposed as
$f=\tau_c \circ \psi$ for a homomorphism $\psi$ and
a translation map $\tau_c$ by $c\in E(\var{k})$.
Fix a rational point $c'\in E(\var{k})$ such that $[\deg(g)](c')=c$
and consider the translation map $\tau_{c'}$,
where $[\deg(g)]$ is the multiplication by $\deg(g)$.
We set $f'=\tau_{c'} \circ\psi$.
Then we have the following equalities.
\begin{align*}
&\hphantom{=} f \circ g\circ h
= \tau_c \circ \psi \circ g \circ \hat{g}\\
&= \tau_c \circ \psi \circ [\deg(g)]
= \tau_c \circ [\deg(g)] \circ \psi\\
&= [\deg(g)]\circ \tau_{c'} \circ \psi
= g\circ h \circ f'.
\end{align*}
This is what we want.
\end{proof}

\begin{prop}\label{prop:reduction elliptic}
Let $ \mathcal{E}$ be a locally free sheaf of rank $2$ on 
a genus one curve $C$ and $X = {\mathbb{P}}( \mathcal{E})$.
Suppose Conjecture \ref{KS} holds for any non-trivial endomorphism on $X$ with
$ \mathcal{E}= \mathcal{O}_{C} \oplus \mathcal{L}$
where $ \mathcal{L}$ is a  line bundle 
of degree zero on $C$.
Then Conjecture \ref{KS} holds for any non-trivial endomorphism
on $X=\P(\mathcal{E})$ for any $\mathcal{E}$.
\end{prop}

\begin{proof}%[Proof of Theorem \ref{Theorem:MainTheorem} in the case of $\P^1$-bundle with $g(C)=1$]
By Lemma \ref{Lemma:FiberPreserving} and Lemma \ref{Lemma:iterate},
we may assume that $f$ preserves fibers.
We can prove Conjecture\ref{KS} in the case of $\deg(f |_F)=1$
by the same way as in the case of $g(C)=0$ since $\deg(f |_F)=1\leq \deg(f_C)$.
Since we are considering the case of $g(C)=1$,
if $\mathcal{E}$ is indecomposable, then $\mathcal{E}$ is semistable
(see \cite[10.2 (c), 10.49]{Mukai} or \cite[V. Exercise 2.8 (c)]{AG}).
By Lemma \ref{Lemma:FiniteBaseChange},
if $\deg(f |_F)>1$ and $\mathcal{E}$ is indecomposable,
there is a finite \'etale covering $g\colon E\longrightarrow C$ satisfying that
$E\times _C X\cong \P(\mathcal{O}_E \oplus \mathcal{L})$
for an invertible sheaf $\mathcal{L}$ over $E$.
Furthermore, by Lemma \ref{Lemma:EndomorphismOnTrivializedFibration},
we can take $E$ equal to $C$ and
there is an endomorphism $f_C'\colon C\longrightarrow C$
satisfying $f_C\circ g=g\circ f_C'$.
Then by the universality of cartesian product $X\times_{C,g}C$,
an endomorphism $f'\colon X\times _{C,g}C\longrightarrow X\times _{C,g}C$ is induced.
By Lemma \ref{Lemma:ReductionByFiniteMorphisms},
it is enough to prove Conjecture \ref{KS} for the endomorphism $f'$.
Thus, we may assume that $\mathcal{E}$ is decomposable, i.e.,
$X\cong  \P(\mathcal{O}_C\oplus \mathcal{L})$.
Then the invariant $e$ is non-negative
by Lemma \ref{Lemma:PositivityOfInvariant}.
When $e$ is positive,
by the same way as the proof of Theorem \ref{Theorem:MainTheorem}
in the case of $g(C)=0$,
the proof is complete.
When $e=0$, we have $\deg \mathcal{L}=0$ and the assertion holds by the assumption.
\end{proof}
In the rest of this subsection, we keep the following notation.
Let $C$ be a genus one curve and $ \mathcal{L}$ an invertible sheaf on $C$ with degree $0$.
Let $X= {\mathbb{P}}( \mathcal{O}_{C} \oplus \mathcal{L})= {\rm Proj} ({\rm Sym}( \mathcal{O}_{C} \oplus \mathcal{L}))$
and $\pi \colon X \longrightarrow C$ the projection.
When $ \mathcal{L}$ is trivial, we have $X \cong C \times {\mathbb{P}}^{1}$,
and by \cite[Theorem1.3]{sano1}, Conjecture \ref{KS} is true for $X$.
Thus we may assume $ \mathcal{L}$ is non-trivial.
In this case, we have two sections of $\pi \colon X \longrightarrow C$ 
corresponding to the projections
$\mathcal{O}_{C} \oplus \mathcal{L} \longrightarrow \mathcal{L}$ 
and $\mathcal{O}_{C} \oplus \mathcal{L} \longrightarrow \mathcal{O}_{C}$.
Let $C_{0}$ and $C_{1}$ denote the images of these sections.
Then we have $ \mathcal{O}_{X}(C_{0})= \mathcal{O}_{X}(1)$ and
$ \mathcal{O}_{X}(C_{1})= \mathcal{O}_{X}(1) {\otimes} \pi^{*} \mathcal{L}^{-1}$.
Since $ \mathcal{L}$ is non-trivial, we have $C_{0}\neq C_{1}$.
But since $\deg \mathcal{L}=0$, $C_{0}$ and
$C_{1}$ are numerically equivalent.
Thus $(C_{0}\cdot C_{1})=(C_{0}^{2})=0$ and therefore $C_{0}\cap C_{1}=\emptyset$.

Let $f$ be a non-trivial endomorphism on $X$ such that
there is a surjective endomorphism $f_{C} \colon C\longrightarrow C$ 
with $\pi \circ f=f_{C} \circ \pi$.

\begin{lem}\label{Lemma:torsion case}
When $ \mathcal{L}$ is a torsion element of $\Pic C$, Conjecture \ref{KS} holds for $f$.
\end{lem}
\begin{proof}
We fix an algebraic group structure on $C$.
Since $ \mathcal{L}$ is torsion, there exists a positive integer $n>0$ such that $[n]^{*} \mathcal{L} \cong \mathcal{O}_{C}$.
Then the base change of $\pi \colon X \longrightarrow C$ 
by $[n] \colon C\longrightarrow C$ is the trivial $\mathbb P^1$-bundle 
${\mathbb{P}}^{1} \times C \longrightarrow C$.
Applying Lemma \ref{Lemma:EndomorphismOnTrivializedFibration} to $g=[n]$, 
we get a finite morphism $h \colon C\longrightarrow C$ such that the base change of 
$\pi \colon X \longrightarrow C$ by $h \colon C\longrightarrow C$ is 
$ {\mathbb{P}}^{1} \times C \longrightarrow C$
and there exists a finite morphism $f_{C}' \colon C \longrightarrow C$ 
with $f_{C} \circ h=h\circ f_{C}'$.
Then $f$ induces a non-trivial endomorphism 
$f' \colon {\mathbb{P}}^{1} \times C \longrightarrow {\mathbb{P}}^{1} \times C$.
By \cite[Theorem1.3]{sano1}, Conjecture \ref{KS} holds for $f'$.
By Lemma \ref{Lemma:ReductionByFiniteMorphisms}, Conjecture \ref{KS} holds also for $f$.
\end{proof}

Now, let $F$ be the numerical class of a fiber of $\pi$.
By Lemma \ref{Lemma:CulculationOfDynamicalDegree},
we have
\begin{align*}
&f^{*}F \equiv aF,\\
&f^{*}C_{0} \equiv bC_{0}
\end{align*}
for some integers $a,b\geq1$.
Note that $a=\deg f_C$, $b=\deg f|_F$ and $ab=\deg f$ (cf.~Lemma \ref{degrees}).

\begin{lem}\label{Lemma: images of C_{i}} \ 
\begin{enumerate}
\item[\rm (1)] 
One of the equalities $f(C_{0})=C_{0}$, $f(C_{0})=C_{1}$ and
$f(C_{0})\cap C_0=f(C_{0})\cap C_1=\emptyset$ 
holds.
The same is true for $f(C_{1})$.
\item[\rm (2)] If $f(C_{0})\cap C_{i}=\emptyset$ for $i=0,1$, then the base change of 
$\pi \colon X \longrightarrow C$
by $f_{C} \colon C \longrightarrow C$ is isomorphic to $ {\mathbb{P}}^{1} \times C$.
In particular, $f_{C}^{*} \mathcal{L} \cong \mathcal{O}_{C}$ and $ \mathcal{L}$ is a torsion element of $\Pic C$.
The same conclusion holds under the assumption that 
$f(C_{1})\cap C_{i}=\emptyset$ for $i=0,1$.
\end{enumerate}
\end{lem}
\begin{proof}
(1) Since $f^{*}C_{i} \equiv bC_{i}$, $C_{0}\equiv C_{1}$ and $(C_{0}^{2})=0$, 
we have $(f_{*}C_{i}\cdot C_{j})=0$ for every $i$ and $j$.
Thus the  assertion follows.

(2) Assume $f(C_{0})\cap C_{i}=\emptyset$ for $i=0,1$.
Consider the following Cartesian diagram.
\[
\xymatrix{
Y \ar[r]^{g} \ar[d]_{\pi'} & X \ar[d]^{\pi}\\
C \ar[r]^{f_{C}} & C
}
\]
Then $Y$ is a $\mathbb P^1$-bundle over $C$ 
associated with the vector bundle $ \mathcal{O}_{C} \oplus f_{C}^{*} \mathcal{L}$.
The pull-backs $C_{i}=g^{-1}(C_{i}), i=0,1$ are sections of $\pi'$.
By the projection formula, we have $(C_{i}'^{2})=0$.
Let $\sigma \colon C \longrightarrow X$ be the section with $\sigma (C) = C_{0}$.
Since $\pi \circ f \circ \sigma=f_{C}$, we get a section 
$s \colon C \longrightarrow Y$ of $\pi'$.
\[
\xymatrix{
&C \ar[ldd]_{s} \ar[d]^{\sigma} \ar@/_11mm/[lddd]_{\rm id}\\
&X \ar[d]^{f}\\
Y \ar[d]^{\pi'} \ar[r]^{g} & X \ar[d]^{\pi}\\
C \ar[r]_{f_{C}} &C
}
\]
Note that $g(s(C))=f(C_{0}) \neq C_{0}, C_{1}$.
Thus $s(C), C_{0}', C_{1}'$ are distinct sections of $\pi'$.
Moreover, by the projection formula, we have $(s(C)\cdot C_{0}')=0$.
Thus we have three sections which are numerically equivalent to each other.
Then Lemma \ref{lem_three_sections} implies 
$f_{C}^{*} \mathcal{L} \cong \mathcal{O}_{C}$ and $Y \cong {\mathbb{P}}^{1} \times C$.
Since $f_{C}^{*}\colon \Pic^{0}C\longrightarrow \Pic^{0}C$ is an isogeny, the kernel of $f_{C}^{*}$ is finite and thus $ \mathcal{L}$ is a torsion element of $\Pic C$.
 \end{proof}

\begin{lem}\label{Lemma: images of C_{i}2} \ 
\begin{enumerate}
\item[\rm (1)] Suppose that
\begin{itemize}
\item $ \mathcal{L}$ is non-torsion in $\Pic C$, 
\item $f(C_{0})=C_{0}\  \text{or}\ C_{1}$, and 
\item $f(C_{1})=C_{0}\  \text{or}\ C_{1}$.
\end{itemize}
Then $f(C_{0})=C_{0}$ and $f(C_{1})=C_{1}$, or
$f(C_{0})=C_{1}$ and $f(C_{1})=C_{0}$.
\item[\rm (2)] If the equalities $f(C_{0})=C_{0}$ and $f(C_{1})=C_{1}$ hold,
then $f^{*}C_{i} \sim_{ {\mathbb{Q}}} bC_{i}$ for $i=0$ and $1$.
\end{enumerate}
\end{lem}
\begin{proof}
(1) Assume that $f(C_{0})=C_{0}$ and $f(C_{1})=C_{0}$.
Then $f_{*}C_{0}=aC_{0}$ and $f_{*}C_{1}=aC_{0}$ as cycles.
Since $f_{C}^{*} \colon \Pic^{0}C \longrightarrow \Pic^{0}C$ is surjective, 
there exists a degree zero divisor $M$ on $C$ such that
$f_{C}^{*} \mathcal{O}_{C}(M) \cong \mathcal{L}$.
Then $C_{1} \sim C_{0}-\pi^{*}f_{C}^{*}M$.
Hence
\[
aC_{0}=f_{*}C_{1} \sim (f_{*}C_{0}-f_{*}\pi^{*}f_{C}^{*}M)=(aC_{0}-f_{*}\pi^{*}f_{C}^{*}M)
\]
and
\[
0 \sim f_{*}\pi^{*}f_{C}^{*}M \sim f_{*}f^{*}\pi^{*}M\sim (\deg f) \pi^{*}M.
\]
Thus $\pi^{*}M$ is torsion and so is $M$.
This implies that $ \mathcal{L}$ is torsion, which contradicts the assumption.

The same argument shows that  the case when $f(C_{0})=C_{1}$ and $f(C_{1})=C_{1}$ 
does not occur.

(2) In this case, we have $f_{*}C_{0} \sim aC_{0}$.
We can write $f^{*}C_{0} \sim bC_{0}+\pi^{*}D$ for some degree zero divisor $D$ on $C$.
Thus
\[
(\deg f)C_{0} \sim f_{*}f^{*}C_{0} \sim abC_{0}+f_{*}\pi^{*}D=(\deg f)C_{0}+f_{*}\pi^{*}D
\]
and $f_{*}\pi^{*}D\sim 0$.
Since $f_{C}^{*} \colon \Pic^{0}C \longrightarrow \Pic^{0}C$ is surjective, there exists a degree zero divisor $D'$ on $C$
such that $f_{C}^{*}D' \sim D$.
Then
\[
0\sim f_{*}\pi^{*}D\sim f_{*}\pi^{*}f_{C}^{*}D' \sim f_{*}f^{*}\pi^{*}D' \sim (\deg f)\pi^{*}D'.
\]
Hence $\pi^{*}D' \sim_{ {\mathbb{Q}}}0$ and $D' \sim _{ {\mathbb{Q}}} 0$.
Therefore $D \sim_{ {\mathbb{Q}}} 0$ and $f^{*}C_{0} \sim_{ {\mathbb{Q}}} bC_{0}$.

Similarly, we have $f^\ast C_1 \sim _\Q bC_1$.
\end{proof}

\begin{lem}\label{Lemma: canonical height zero}
Suppose $a<b$.
If $f^{*}C_{i} \sim_{ {\mathbb{Q}}} bC_{i}$ for $i=0,1$, the line bundle $\mathcal{L}$ is a torsion element of $\Pic C$.
\end{lem}
\begin{proof}
Let $L$ be a divisor on $C$ such that $ {\mathcal{O}}_{C}(L) \cong \mathcal{L}$.
Note that $C_{1}\sim C_{0}-\pi^{*}L$.
Thus
\[
f^{*}\pi^{*}L\sim f^{*}(C_{0}-C_{1})\sim_\mathbb Q bC_{0}-bC_{1} \sim b\pi^{*}L
\]
and $f_{C}^{*}L\sim_\mathbb Q bL$ hold.

Thus, from the following lemma, $\mathcal{L}$ is a torsion element.
\end{proof}

\begin{lem}\label{prop:torsion}
Let $a,b$ be integers such that $1\leq a<b$.
Let $C$ be a curve of genus one defined over an algebraically closed field $k$.
Let $f_C\colon C \longrightarrow C$ be an endomorphism of $\deg f_C=a$.
If $L$ is a divisor on $C$ of degree $0$ satisfying
\[
f_C^\ast L\sim_\Q bL,
\]
the divisor $L$ is a torsion element of $\Pic^0(C)$
\end{lem}

\begin{proof}
By the definition of $\Q$-linear equivalence,
we have $f_C^\ast rL \sim brL$ for some positive integer $r$.
Since the curve $C$ is of genus one, the group $\Pic^0(C)$ is an elliptic curve.
Assume the (group) endomorphism
\[
f_C^\ast -[b]\colon \Pic^0(C)\longrightarrow \Pic^0(C)
\]
is the $0$ map.
Then we have the equalities $a=\deg f_C =\deg f_C^\ast =\deg [b]=b^2$.
But this contradicts to the inequality $1\leq a <b$.
Hence the map $f_C^\ast -[b]$ is an isogeny,
and $\Ker (f_C^\ast -[b])\subset \Pic^0(C)$ is a finite group scheme.
In particular, the order of $rL \in \Ker (f_C^\ast -[b])(k)$ is finite.
Thus, $L$ is a torsion element.
\end{proof}

\begin{rmk}
We can actually prove the following.
Let $X$ be a smooth projective variety over $\overline{\mathbb Q}$
and $f \colon X \longrightarrow X$ be a surjective morphism over $\overline{\mathbb Q}$
with first dynamical degree $\delta$.
If an $\R$-divisor $D$ on $X$ satisfies
\[
f^{*}D \sim_{\R} \lambda D
\]
for some $\lambda>\delta$, then one has $D \sim_{\R} 0$.
\begin{proof}[Sketch of the proof]
Consider the canonical height 
\[
\hat{h}_{D}(P)=\lim_{n \to \infty}h_{D}(f^{n}(P))/\lambda^{n}
\]
where $h_{D}$ is a height associated with $D$ (cf.\ \cite{callsilv}).
If $\hat{h}_{D}(P)\neq 0$ for some $P$, then we can prove $ \overline{\alpha}_{f}(P) \geq \lambda$.
This contradicts to the fact $\delta \geq \overline{\alpha}_{f}(P)$ and the assumption $\lambda > \delta$.
Thus one has $\hat{h}_{D}=0$ and therefore $h_{D}=\hat{h}_{D}+O(1)=O(1)$.
By a theorem of Serre, we get $D \sim_{\R} 0$. 
\end{proof}
\end{rmk}

\begin{prop}\label{prop:split deg zero case}
Let $ \mathcal{L}$ be an invertible sheaf of degree zero on 
a genus one curve $C$ and $X= {\mathbb{P}}( \mathcal{O}_{C}\oplus \mathcal{L})$.
For any non-trivial endomorphism $f \colon X \longrightarrow X$, Conjecture \ref{KS} holds.
\end{prop}

\begin{proof}
By Lemma \ref{Lemma:torsion case} and Proposition \ref{prop:torsion} we may assume $a \geq b$.
In this case, $\delta_{f}=a$ and Conjecture \ref{KS} can be proved
as in the proof of Proposition \ref{prop:Hirz}.
\end{proof}

\begin{proof}[Proof of Theorem \ref{Theorem:MainTheorem} for $\P^1$-bundles over genus one curves]
As we argued at the first of Section \ref{Section:EndomorphismsOnSurfaces},
we may assume that the endomorphism $f\colon X\longrightarrow X$ is not an automorphism.
Then the assertion follows from Proposition \ref{prop:reduction elliptic} and Proposition \ref{prop:split deg zero case}.
\end{proof}

\begin{rmk}
In the above setting, the line bundle $ \mathcal{L}$ is actually an eigenvector for $f_{C}^{*}$
up to linear equivalence.
More precisely, for a $\P^1$-bundle 
$\pi \colon X= {\mathbb{P}}( \mathcal{O}_{C} \oplus \mathcal{L}) \longrightarrow C$
over a curve $C$ with $\deg \mathcal{L}=0$ and an endomorphism 
$f \colon X\longrightarrow X$ that induces an endomorphism 
$f_{C}\colon C \longrightarrow C$,
there exists an integer $t$ such that $ f_{C}^{*} \mathcal{L} \cong \mathcal{L}^{t}$.
Indeed, let $C_{0}$ and $C_{1}$ be the sections defined above.
Since $(f^{*}(C_{0})\cdot C_{0})=0$, we can write 
$ \mathcal{O}_{X}(f^{-1}(C_{0})) \cong \mathcal{O}_{X}(mC_{0}) {\otimes} \pi^{*} \mathcal{N}$
for some integer $m$ and degree zero line bundle $ \mathcal{N}$ on $C$.
Since
\begin{align*}
0&\neq  H^{0}( \mathcal{O}_{X}(f^{-1}(C_{0}))) = H^{0}(\mathcal{O}_{X}(mC_{0}) {\otimes} \pi^{*} \mathcal{N})\\
&=H^{0}(\Sym^{m}( \mathcal{O}_{C}\oplus \mathcal{L}) {\otimes} \mathcal{N})=\bigoplus_{i=0}^{m}H^{0}( \mathcal{L}^{i} {\otimes} \mathcal{N}),
\end{align*}
we have $\mathcal{N} \cong \mathcal{L}^{r}$ for some $-m \leq r \leq0$.
Thus $f^{*} \mathcal{O}_{X}(C_{0}) \cong \mathcal{O}_{X}(mC_{0}) {\otimes}\pi^{*} \mathcal{L}^{r}$.
The key is the calculation of global sections using projection formula.
Since $ \mathcal{O}_{X}(C_{1}) \cong \mathcal{O}_{X}(C_{0}) {\otimes} \pi^{*} \mathcal{L}^{-1}$,
we have $\pi_{*} \mathcal{O}_{X}(mC_{1}) \cong \pi_{*} \mathcal{O}_{X}(mC_{0}) {\otimes} \mathcal{L}^{-m}$.
Moreover, since $C_{0}$ and $C_{1}$ are numerically equivalent, we can similarly get 
$f^{*} \mathcal{O}_{X}(C_{1}) \cong \mathcal{O}_{X}(mC_{0}) {\otimes} \pi^{*} \mathcal{L}^{s}$ for some integer $s$.
Thus, $f^{*}\pi^{*} \mathcal{L} \cong \pi^{*} \mathcal{L}^{r-s}$.
Therefore, $\pi^{*}f_{C}^{*} \mathcal{L} \cong \pi^{*} \mathcal{L}^{r-s}$.
Since $\pi^{*} \colon \Pic C \longrightarrow \Pic X$ is injective, we get $f_{C}^{*} \mathcal{L} \cong \mathcal{L}^{r-s}$.
\end{rmk}

\subsection{$\mathbb P^1$-bundles over curves of genus $\geq 2$}
\label{Subsection:general type base curve}

By the following proposition, Conjecture \ref{KS} trivially holds in this case.

\begin{prop}
Let $C$ be a curve with $g(C)\geq2$ and 
$\pi \colon X\longrightarrow C$ be a $ {\mathbb{P}}^{1}$-bundle over $C$.
Let $f \colon X \longrightarrow X$ be a surjective endomorphism.
Then there exists an integer $t > 0$ such that $f^{t}$ is a morphism over $C$, 
that is, $f^t$ satisfies $\pi \circ f^{t}=\pi$. 
In particular, $f$ admits no Zariski dense orbit.
\end{prop}

\begin{proof}
By Lemma \ref{Lemma:FiberPreserving}, we may assume that $f$ induces
a surjective endomorphism $f_{C} \colon C \longrightarrow C$ with 
$\pi \circ f=f_{C} \circ \pi$.
Since $C$ is of general type, $f_{C}$ is an automorphism of finite order and the assertion follows.
\end{proof}

\begin{rmk}
The fact that $f$ does not admit any Zariski dense orbits also follows
from the Mordell conjecture (Faltings's theorem).
Indeed, assume there exists a Zariski dense orbit $\mathcal O_f(P)$ on $X$. 
Then $\pi(\mathcal O_f(P))$ is also Zariski dense in $C$. 
We may assume that $X, C, f, \pi, P$ are defined over a number field $K$.
Since $g(C)\geq2$, by the Mordell conjecture, the set of $K$-rational points $C(K)$ is finite
and therefore $\pi(\mathcal O_f(P))$ is also finite. This is a contradiction.
\end{rmk}

\section{Hyperelliptic surfaces}\label{Section:HyperEllipticSurface}

\begin{thm}
Let $X$ be a hyperelliptic surface and $f \colon X \longrightarrow X$ 
a non-trivial endomorphism on $X$. 
Then Conjecture \ref{KS} holds for $f$. 
\end{thm}

\begin{proof}
Let $\pi \colon X \longrightarrow E$ be the Albanese map of $X$.
By the universality of $\pi$, there is a morphism $g \colon E \longrightarrow E$ 
satisfying $\pi \circ f = g \circ \pi$.
It is well-known that $E$ is a genus one curve, $\pi$ is a surjective morphism 
with connected fibers, and there is an \'etale cover 
$\phi \colon E' \longrightarrow E$ such that 
$X'=X \times_E E' \cong F \times E'$, where $F$ is a genus one curve 
(cf.~\cite[Chapter 10]{Badescu}).
In particular, $X'$ is an abelian surface.
By Lemma \ref{Lemma:EndomorphismOnTrivializedFibration},
taking a further \'etale base change, we may assume that there is an endomorphism 
$h \colon E' \longrightarrow E'$ such that $\phi \circ h=g \circ \phi$.
Let $\pi' \colon X' \longrightarrow E'$ and 
$\psi \colon X' \longrightarrow X$ be the induced morphisms.
Then, by the universality of fiber products, there is a morphism 
$f' \colon X' \longrightarrow X'$ 
satisfying $\pi' \circ f'= \pi' \circ h$ and $\psi \circ f' = f \circ \psi$.
Applying Lemma \ref{Lemma:ReductionByFiniteMorphisms},
it is enough to prove Conjecture \ref{KS} for the endomorphism $f'$.
Since $X'$ is an abelian variety, it holds 
by \cite[Corollary 31]{ab1} and \cite[Theorem 2]{ab2}.
\end{proof}

\section{Surfaces with $\kappa(X)=1$}\label{Section:EllipticSurface}
Let $f \colon X \longrightarrow X$ be a non-trivial endomorphism on a surface $X$ 
with $\kappa(X)=1$.
In this section we shall prove that $f$ does not admit any Zariski dense forward $f$-orbit.
Although this result is a special case of \cite[Theorem A]{NaZh} 
(see Remark \ref{rem_for_KS}),
we will give a simpler proof of it.

By Lemma \ref{lem_min}, $X$ is minimal and $f$ is \'etale. 
Since $\deg(f) \geq 2$, we have $\chi(X, \mathcal O_X)=0$. 

Let $\phi = \phi_{|mK_X|} \colon X \longrightarrow \mathbb P^N=\mathbb P H^0(X, mK_X)$ 
be the Iitaka fibration of $X$ and set $C_0=\phi(X)$. 
Since $f$ is \'etale, it induces an automorphism 
$g \colon \mathbb P^N \longrightarrow \mathbb P^N$ 
such that $\phi \circ f = g \circ \phi$ (cf.~\cite[Lemma 3.1]{FN2}). 
The restriction of $g$ to $C_0$ gives an automorphism 
$f_{C_0}\colon C_0 \longrightarrow C_0$ 
such that $\phi \circ f= f_{C_0} \circ \phi$. 
Take the normalization $\nu \colon C \longrightarrow C_0$ of $C_0$. 
Then $\phi$ factors as 
$X \overset{\pi}{\longrightarrow} C \overset{\nu}{\longrightarrow} C_0$ and 
$\pi$ is an elliptic fibration. 
Moreover, $f_{C_0}$ lifts to an automorphism $f_C\colon C \longrightarrow C$ such that 
$\pi \circ f=f_C \circ \pi$. 

So we obtain an elliptic fibration $\pi \colon X \longrightarrow C$ 
and an automorphism $f_C$ on $C$ 
such that $\pi\circ f=f_C\circ \pi$
In this situation, the following holds.

\begin{thm}\label{inv_of_fibers}
Let $X$ be a surface with $\kappa(X)=1$, 
$\pi \colon X \longrightarrow C$ an elliptic fibration, 
$f \colon X \longrightarrow X$ a non-trivial endomorphism, 
and $f_C \colon C \longrightarrow C$ an automorphism such that 
$\pi \circ f=f_C \circ \pi$. 
Then $f_C^t=\mathrm{id}_C$ for a positive integer $t$. 
\end{thm}

\begin{proof}
Let $\{ P_1, \ldots, P_r \}$ be the points over which the fibers of $\pi$ are 
multiple fibers (possibly $r=0$, i.e.~$\pi$ does not have any multiple fibers).
We denote by $m_i$ denotes the multiplicity of the fiber $\pi^*P_i$ for every $i$. 
Then we have the canonical bundle formula: 
$$K_X = \pi^*(K_C + L)+ \sum_{i=1}^r \frac{m_i-1}{m_i} \pi^*P_i,$$
where $L$ is a divisor on $C$ such that $\deg(L)=\chi(X, \mathcal O_X)$. 
Then $\deg(L)=0$ because $f$ is \'etale and $\deg(f) \geq 2$ (cf.~Lemma \ref{lem_min}). 
Since $\kappa(X)=1$, the divisor $K_C+L+\sum_{i=1}^r \frac{m_i-1}{m_i} P_i$ must have positivedegrees.
So we have 
\begin{equation*}
2(g(C)-1)+\sum_{i=1}^r \frac{m_i-1}{m_i}>0. \tag{$*$}
\end{equation*}

For any $i$, set $Q_i=f_C^{-1}(P_i)$. Then 
$\pi^* Q_i = \pi^* f_C^*P_i =f^*\pi^*P_i$ is a multiple fiber. 
So $(f_C)|_{\{P_1, \ldots, P_r\}}$ is a permutation of $\{ P_1, \ldots, P_r \}$ 
since $f_C$ is an automorphism. 

We divide the proof into three cases according to the genus $g(C)$ of $C$: 

(1) $g(C) \geq 2$; then the automorphism group of $C$ is finite. 
So $f_C^t=\mathrm{id}_C$ for a positive integer $t$. 

(2) $g(C)=1$; by ($*$), it follows that $r \geq 1$. 
For a suitable $t$, all $P_i$ are fixed points of $f_C^t$. 
We put the algebraic group structure on $C$
such that $P_1$ is the identity element. 
Then $f_C^t$ is a group automorphism on $C$.
So $f_C^{ts}= \mathrm{id}_C$ for a suitable $s$
since the group of group automorphisms  on $C$ is finite.

(3) $g(C)=0$; again by ($*$), it follows that $r \geq 3$. 
For a suitable $t$, all $P_i$ are fixed points of $f_C^t$. 
Then $f_C^t$ fixes at least three points, which implies that $f_C^t$ is in fact the 
identity map. 
\end{proof}

Immediately we obtain the following corollary.

\begin{cor}
Let $f \colon X \longrightarrow X$ be a non-trivial endomorphism on a surface $X$ 
with $\kappa(X)=1$. 
Then there does not exist any Zariski dense $f$-orbit. 
\end{cor}

Therefore Conjecture \ref{KS} trivially holds 
for non-trivial endomorphisms on surfaces of Kodaira dimension 1. 

\section{Existence of a rational point $P$ satisfying $\alpha_f(P)=\delta_f$}
\label{Section:ExistenceOfOrbits}
In this section, we prove Theorem \ref{thm_existence} and Theorem \ref{thm_large collection}.
Theorem \ref{thm_existence} follows from the following lemma.
A subset $\Sigma \subset V( \overline{k})$ is called a 
{\it set of bounded height} if for an (every) ample divisor $A$ on $V$,
the height function $h_{A}$ associated with $A$ is a bounded function on $\Sigma$.

\begin{lem}\label{lem:existence}
Let $X$ be a smooth projective variety and $f \colon X \longrightarrow X$ a surjective endomorphism
with $\delta_{f}>1$.
Let $D \not \equiv 0$ be a nef $\R$-divisor such that $f^{*}D \equiv \delta_{f} D$.
\ Let $V \subset X$ be a closed subvariety of positive dimension
such that $(D^{\dim V}\cdot V)>0$.
Then there exists a non-empty open subset $U \subset V$ and a set $\Sigma \subset U( \overline{k})$
of bounded height such that for every $P \in U( \overline{k})\setminus \Sigma$ we have $ \alpha_{f}(P)=\delta_{f}$.
\end{lem}

\begin{rmk}
By Perron-Frobenius-type result of \cite[Theorem]{Birkhoff},
there is a nef $\R$-divisor $D\not \equiv 0$
satisfying the condition $f^\ast D \equiv \delta_f D$
since $f^\ast$ preserves the nef cone.
\end{rmk}

\begin{proof}
Fix a height function $h_{D}$ associated with $D$.
For every $P\in X( \overline{k})$, the following limit exists (cf.~\cite[Theorem 5]{rat}).
\[
\hat{h}(P)=\lim_{n \to \infty} \frac{h_{D}(f^n(P))}{\delta_{f}^{n}}
\]
The function $\hat{h}$ has the following properties (cf.~\cite[Theorem 5]{rat}).
\begin{enumerate}
\item[(i)] $\hat{h}=h_{D}+O(\sqrt{h_{H}})$ where $H$ is any ample divisor on $X$ and $h_{H}\geq1$ is a height function associated with $H$.
\item[(ii)] If $\hat{h}(P)>0$, then $ \alpha_{f}(P)=\delta_{f}$.
\end{enumerate}

Since $(D^{\dim V}\cdot V)>0$, we have $({D|_{V}}^{\dim V})>0$ and $D|_{V}$ is big.
Thus we can write $D|_{V} \sim_{\R} A+E$ with an ample $\R$-divisor $A$ and an effective $\R$-divisor $E$ on $V$.
Therefore we have
\[
\hat{h}|_{V( \overline{k})}=h_{A}+h_{E}+O(\sqrt{h_{A}})
\]
where $h_{A}, h_{E}$ are height functions associated with $A,E$ and $h_{A}$ is taken to be $h_{A}\geq 1$.
In particular, there exists a positive real number $B>0$ such that
$h_{A}+h_{E}-\hat{h}|_{V( \overline{k})}\leq B \sqrt{h_{A}}$.
Then we have the following inclusions.
\begin{align*}
\{ P\in V( \overline{k})\mid \hat{h}(P)\leq0\}
&\subset \{ P\in V( \overline{k})\mid  h_{A}(P)+h_{E}(P)\leq B\sqrt{h_{A}(P)}\}\\
&\subset \Supp E \cup \{ P\in V( \overline{k})\mid  h_{A}(P)\leq B\sqrt{h_{A}(P)}\}\\
&= \Supp E \cup \{ P\in V( \overline{k})\mid  h_{A}(P)\leq B^{2}\}.
\end{align*}
Hence we can take $U= V \setminus \Supp E$ and $\Sigma=\{ P\in U( \overline{k})\mid \hat{h}(P)\leq0\}$.
\end{proof}

\begin{cor}\label{cor:pos curve}
Let $X$ be a smooth projective variety of dimension $N$
and $f \colon X \longrightarrow X$ a surjective endomorphism.
Let $C$ be a irreducible curve which is a
complete intersection of ample effective divisors $H_{1},\ldots ,H_{N-1}$  on $X$.
Then for infinitely many points $P$ on $C$, we have $ \alpha_{f}(P)=\delta_{f}$.
\end{cor}
\begin{proof}
We may assume $\delta_{f}>1$.
Let $D$ be as in Lemma \ref{lem:existence}.
Then $(D\cdot C)=(D\cdot H_{1}\cdots H_{N-1})>0$ (cf.~\cite[Lemma 20]{rat}).
Since $C( \overline{k})$ is not a set of bounded height, the assertion follows from Lemma \ref{lem:existence}.
\end{proof}

To prove Theorem \ref{thm_large collection}, we need the following theorem
which is a corollary of the dynamical Mordell--Lang conjecture for \'etale finite morphisms. 

\begin{thm}[Bell--Ghioca--Tucker {\cite[Corollary 1.4]{bgt}}]\label{dml}
Let $f \colon X \longrightarrow X$ be an \'etale finite morphism of smooth projective variety $X$.
Let $P \in X( \overline{k})$.
If the orbit $ \mathcal{O}_{f}(P)$ is Zariski dense in $X$, then any proper closed subvariety
of $X$ intersects $ \mathcal{O}_{f}(P)$ in at most finitely many points.
\end{thm}

\begin{proof}[Proof of Theorem \ref{thm_large collection}]
We may assume $\dim X \geq2$.
Since we are working over $ \overline{k}$, we can write the set of 
all proper subvarieties of $X$ as
\[
\{V_{i} \subsetneq X\mid i=0,1,2,\ldots\}.
\]
By Corollary \ref{cor:pos curve}, we can take a point $P_{0}\in X\setminus V_{0}$
such that $ \alpha_{f}(P)=\delta_{f}$.
Assume we can construct $P_{0},\ldots ,P_{n}$ satisfying the following conditions.
\begin{enumerate}
\item $ \alpha_{f}(P_{i})=\delta_{f}$ for  $i=0,\ldots ,n$.
\item $ \mathcal{O}_{f}(P_{i}) \cap \mathcal{O}_{f}(P_{j})=\emptyset$ for $i\neq j$.
\item $P_{i} \notin V_{i}$ for  $i=0,\ldots ,n$.
\end{enumerate}
Now, take a complete intersection curve $C \subset X$
satisfying the following conditions.
\begin{itemize}
\item For $i=0,\ldots, n$, $C \not \subset \mathcal{O}_{f}(P_{i})$ if $ \overline{ \mathcal{ O}_{f}(P_{i})} \neq X$.
\item For $i=0,\ldots, n$, $C \not \subset \mathcal{O}_{f^{-1}}(P_{i})$ if $ \overline{ \mathcal{ O}_{f^{-1}}(P_{i})} \neq X$.
\item $C \not \subset V_{n+1}$.
\end{itemize}
By Theorem \ref{dml}, if $ \mathcal{O}_{f^{\pm}}(P_{i})$ is Zariski dense in $X$, then $ \mathcal{O}_{f^{\pm}}(P_{i}) \cap C$
is a finite set.
By Corollary \ref{cor:pos curve}, there exists a point
$$P_{n+1}\in C \setminus \left(\bigcup_{0\leq i\leq n} \mathcal{O}_{f}(P_{i}) \cup \bigcup_{0\leq i\leq n} \mathcal{O}_{f^{-1}}(P_{i})\cup V_{n+1}\right) $$
such that $ \alpha_{f}(P_{n+1})=\delta_{f}$.
Then $P_{0},\ldots, P_{n+1}$ satisfy the same conditions.
Therefore we get a subset $S=\{P_{i}\mid i=0,1,2,\ldots\}$ of $X$ which satisfies the desired conditions.
\end{proof}

\section*{Acknowledgements}
The authors would like to thank 
Professors Tetsushi Ito, Osamu Fujino, and  Tomohide Terasoma for helpful advice.
They would also like to thank Takeru Fukuoka and Hiroyasu Miyazaki 
for answering their questions.

\end{document}